\pdfoutput=1 
\documentclass[a4paper]{article}
\usepackage{amsthm,amssymb,amsmath,enumerate,graphicx}

\newtheorem{THM}{Theorem}[section]
\newtheorem{LEM}[THM]{Lemma}

\newcommand\abs[1]{\left\lvert #1\right\rvert}
\def\ord(#1){\abs{#1}}
\def\shift(#1)(#2){\!\!\downarrow\!{}^{#1}_{\raise .01ex\vbox to 0pt{\vss\hbox{$\scriptstyle #2$}}}\,}
\def\ucl(#1){\lfloor #1 \rfloor}
\def\dcl(#1){\lceil #1 \rceil}
\def\specrel#1#2{\mathrel{\mathop{\kern0pt #1}\limits_{#2}}}

\def\F{\mathcal F}

\def\lowfwd #1#2#3{{\mathop{\kern0pt #1}\limits^{\kern#2pt\raise.#3ex
\vbox to 0pt{\hbox{$\scriptscriptstyle\rightarrow$}\vss}}}}
\def\lowbkwd #1#2#3{{\mathop{\kern0pt #1}\limits^{\kern#2pt\raise.#3ex
\vbox to 0pt{\hbox{$\scriptscriptstyle\leftarrow$}\vss}}}}

\def\ve{\kern-1pt\lowfwd e{1.5}2\kern-1pt}
\def\ev{\lowbkwd e02}
\def\vedash{{\mathop{\kern0pt e\lower.5pt\hbox{${}
     \scriptstyle'$}}\limits^{\kern0pt\raise.02ex
     \vbox to 0pt{\hbox{$\scriptscriptstyle\rightarrow$}\vss}}}}
\def\evdash{{\mathop{\kern0pt e\lower.5pt\hbox{${}
     \scriptstyle'$}}\limits^{\kern0pt\raise.02ex
     \vbox to 0pt{\hbox{$\scriptscriptstyle\leftarrow$}\vss}}}}

\def\vex{\lowfwd {e_x}02}
\def\evx{\lowbkwd {e_x}02}
\def\veminus{{\mathop{\kern0pt e\raise3pt\hbox{${}
     \scriptscriptstyle-$}}\limits^{\kern0pt\lower2pt
     \vbox to 0pt{\hbox{$\scriptscriptstyle\rightarrow$}\vss}}}}
\def\evminus{{\mathop{\kern0pt e\raise3pt\hbox{${}
     \scriptscriptstyle-$}}\limits^{\kern0pt\lower2pt
     \vbox to 0pt{\hbox{$\scriptscriptstyle\leftarrow$}\vss}}}}
\def\vf{\kern-1pt\lowfwd f{1.5}2\kern-1pt}
\def\fv{\lowbkwd f02}
\def\vr{\lowfwd r{1.5}2}
\def\rv{\lowbkwd r02}
\def\vro{\lowfwd {r_0}12}
\def\rvo{\lowbkwd {r_0}02}
\def\vri{\lowfwd {r_i}12}
\def\rvi{\lowbkwd {r_i}02}

\def\vrdash{{\mathop{\kern0pt r\lower.5pt\hbox{${}
     \scriptstyle'$}}\limits^{\kern0pt\raise.02ex
     \vbox to 0pt{\hbox{$\scriptscriptstyle\rightarrow$}\vss}}}}
\def\rvdash{{\mathop{\kern0pt r\lower.5pt\hbox{${}
     \scriptstyle'$}}\limits^{\kern0pt\raise.02ex
     \vbox to 0pt{\hbox{$\scriptscriptstyle\leftarrow$}\vss}}}}
\def\vrddash{{\mathop{\kern0pt r\lower.5pt\hbox{${}
     \scriptstyle''$}}\limits^{\kern0pt\raise.02ex
     \vbox to 0pt{\hbox{$\scriptscriptstyle\rightarrow$}\vss}}}}
\def\rvddash{{\mathop{\kern0pt r\lower.5pt\hbox{${}
     \scriptstyle''$}}\limits^{\kern0pt\raise.02ex
     \vbox to 0pt{\hbox{$\scriptscriptstyle\leftarrow$}\vss}}}}
\def\vrone{\lowfwd {r_1}02}
\def\rvone{\lowbkwd {r_1}{-1}2}
\def\vrtwo{\lowfwd {r_2}12}
\def\rvtwo{\lowbkwd {r_2}02}
\def\vs{\lowfwd s{1.5}1}
\def\sv{\lowbkwd s{1.5}1}
\def\vso{\lowfwd {s_0}11}
\def\svo{\lowbkwd {s_0}02}
\def\vsone{\lowfwd {s_1}11}
\def\svone{\lowbkwd {s_1}02}

\def\vsi{\lowfwd {s_i}11}
\def\vsidash{{\mathop{\kern0pt s_i\kern-3.5pt\lower.3pt\hbox{${}
     \scriptstyle'$}}\limits^{\kern0pt\raise.02ex
     \vbox to 0pt{\hbox{$\scriptscriptstyle\rightarrow$}\vss}}}}
\def\svj{\lowbkwd {s_j}02}
\def\vS{{\vec S}} 
\def\vSr{{\vec S}_{\raise.1ex\vbox to 0pt{\vss\hbox{$\scriptstyle\ge\vr$}}}}
\def\vSrone{{\vec S}_{\raise.1ex\vbox to 0pt{\vss\hbox{$\scriptstyle\ge\vrone$}}}}
\def\vSdash{{\mathop{\kern0pt S\lower-1pt\hbox{${}
     \scriptstyle'$}}\limits^{\kern2pt\raise.1ex
     \vbox to 0pt{\hbox{$\scriptscriptstyle\rightarrow$}\vss}}}}
\def\vsdash{{\mathop{\kern0pt s\lower.5pt\hbox{${}
     \scriptstyle'$}}\limits^{\kern0pt\raise.02ex
     \vbox to 0pt{\hbox{$\scriptscriptstyle\rightarrow$}\vss}}}}
\def\svdash{{\mathop{\kern0pt s\lower.5pt\hbox{${}
     \scriptstyle'$}}\limits^{\kern0pt\raise.02ex
     \vbox to 0pt{\hbox{$\scriptscriptstyle\leftarrow$}\vss}}}}

\def\vU{{\vec U}} 

\def\sub{\subseteq}
\def\supe{\supseteq}
\def\sm{\smallsetminus}
\def\td{tree-decom\-po\-si\-tion}

\newcommand\COMMENT[1]{}

\def\?#1{\vadjust{\vbox to 0pt{\vss\vskip-8pt\leftline{%
     \llap{\hbox{\vbox{\pretolerance=-1
     \doublehyphendemerits=0\finalhyphendemerits=0
     \hsize16truemm\tolerance=10000\small
     \lineskip=0pt\lineskiplimit=0pt
     \rightskip=0pt plus16truemm\baselineskip8pt\noindent
     \hskip0pt        
     #1\endgraf}\hskip7truemm}}}\vss}}}

\title{Tangle-tree duality in abstract separation systems}
\author{Reinhard Diestel\\
  Mathematisches Seminar, Universit\"at Hamburg
  \and 
  Sang-il Oum%
  \\
  Discrete Mathematics Group, Institute for Basic Science (IBS)\\
  Department of Mathematical Sciences, KAIST}

\begin{document}
\abovedisplayshortskip=-3pt plus3pt
\belowdisplayshortskip=6pt

\maketitle

\begin{abstract}\noindent
  We prove a general width duality theorem for combinatorial structures with well-defined notions of cohesion and separation. These might be graphs or matroids, but can be much more general or quite different. The theorem asserts a duality between the existence of high cohesion somewhere local and a global overall tree structure.

We describe cohesive substructures in a unified way in the format of tangles: as orientations of low-order separations satisfying certain consistency axioms. 
   These axioms can be expressed without reference to the underlying structure, such as a graph or matroid, but just in terms of the poset of the separations themselves. This makes it possible to identify tangles, and apply our tangle-tree duality theorem, in very diverse settings.\looseness=-1

Our result implies all the classical duality theorems for width parameters in graph minor theory, such as path-width, tree-width, branch-width or rank-width. It yields new, tangle-type, duality theorems for tree-width and path-width. It implies the existence of width parameters dual to cohesive substructures such as $k$-blocks, edge-tangles, or given subsets of tangles, for which no width duality theorems were previously known.

Abstract separation systems can be found also in structures quite unlike graphs and matroids. For example, our theorem can be applied to image analysis by capturing the regions of an image as tangles of separations defined as natural partitions of its set of pixels. It can be applied in big data contexts by capturing clusters as tangles. It can be applied in the social sciences, e.g.\ by capturing as tangles the few typical mindsets of individuals found by a survey. It could also be applied in pure mathematics, e.g.\ to separations of compact manifolds.
   \end{abstract}

\section{Introduction}\label{sec:intro}

There are a number of theorems in the structure theory of sparse graphs that assert a duality between high connectivity present somewhere in the graph and an overall tree structure. For example, a graph either has a large complete minor or a \td\ into torsos of essentially bounded genus, but not both~\cite{DiestelBook16, GMXVI}. And it either has a large grid minor or a \td\ into parts of bounded size, but not both~\cite{DiestelBook16, GMV}. Let us loosely refer to such highly cohesive substructures of a graph, defined in terms of subsets of its vertices together with some required edges, as {\em concrete\/} highly cohesive substructures~({\em HCS\/}s).\looseness=-1

An example of a concrete HCS for which no dual notion of global tree structure has been known is that of a $k$-{\em block\/} \cite{ForcingBlocks, mader78}: a set of at least $k$ vertices no two of which can be separated in the graph by deleting fewer than~$k$ vertices.%
   \COMMENT{}

Conversely, there are  a number of so-called width parameters for graphs, invariants whose boundedness asserts that the graph has some kind of global tree structure, for which there are no obvious dual concrete HCSs.

Amini, Mazoit, Nisse, and Thomass\'e~\cite{MazoitPartition} addressed this latter problem in a broad way: they showed how to construct, for many width parameters including all the then known ones, dual concrete HCSs akin to brambles (see~\cite{DiestelBook16}).%
 \COMMENT{}
   For each parameter, the existence of such an HCS forces this parameter to be large, and conversely, whenever one of these width parameters is large there exists a concrete bramble-type HCS to witness this.

\medbreak

In one of their seminal papers on graph minors~\cite{GMX}, Robertson and Seymour introduced a very different way to capture high cohesion somewhere in a graph, which they call {\em tangles\/}. The basic idea behind these is as follows. Given a concrete HCS $X$ in a graph~$G$ and a low-order separation,%
 \COMMENT{}
   most of~$X$ will lie on one of its two sides: otherwise $X$ could not be highly cohesive.%
 \COMMENT{}
   In this way~$X$, whatever it is, {\em orients\/} each of the low-order separations of~$G$ towards one of its sides, the side that contains most of~$X$. These orientations of all the low-order separations will be `consistent' in various ways, since they all point towards~$X$: no two of them, for example, will point away from each other.

Robertson and Seymour~\cite{GMX} noticed that these orientations of all the low-order separations of~$G$ captured most of what they needed to know about~$X$. Consequently, they defined a {\em tangle of order~$k$} in a graph as a way to orient all its separations of order~$<k$, consistently in some precise sense not relevant here.\looseness=-1

The notion of a tangle brought with it a shift of paradigm in the connectivity theory of graphs~\cite{ReedConnectivityMeasure}: we can now think of a consistent orientation of all the low-order separations of a graph as a `highly cohesive substructure' in its own right: no longer a concrete one, but an {\em abstract\/}~HCS. Such abstract HCSs, though maybe unfamiliar at first, are often `deeper' than concrete ones, because they pick out only the essential information. But they are also easier to work with: one no longer has to worry about the details, say, of where exactly in the graph a subdivided grid has all its connecting paths. And most importantly, they are able to capture HCSs that are inherently fuzzy. For example, the additional detail that a subdivided grid contains over the tangle it defines is not only superfluous but can be misleading: each individual branch vertex can, and typically will, lie on the {\em wrong\/} side of {\em some\/} low-order separation, the side that does not contain most of the grid. (Consider, for example, the separation defined by the four neighbours of a given vertex in an actual grid.)

Our first aim in this paper is to do for abstract HCSs in graphs and matroids the converse%
   \COMMENT{}
   of what Amini et al.\ did for concrete ones: starting from a unified definition of abstract HCSs, we prove a general duality theorem that describes corresponding tree structures to witness the  nonexistence of these HCSs.

Generalizing the specific notion of a tangle from~\cite{GMX}, we shall define types of abstract HCSs to be called `$\F$-tangles', where $\F$ encodes some particular type of consistency. Thus, an $\F$-tangle in a graph will be a way to orient all its separations of order~$<k$ (for some~$k$) consistently in a sense specified by~$\F$: different notions of consistency will give rise to different sets~$\F$ and result in different $\F$-tangles. But we shall prove one unified duality theorem saying that, for every suitable~$\F$, a~given graph either has an $\F$-tangle or a global tree structure that clearly precludes the existence of an $\F$-tangle.

Our duality theorem will easily imply the two known tangle-type duality theorems from graph minor theory: the classical Robertson-Seymour one for tangles and branch-width in graphs~\cite{GMX}, and its analogue for matroids~\cite{BranchDecMatroids, GMX}. This has been shown in detail in~\cite{TangleTreeGraphsMatroids}.

It will also imply new, tangle-type, dual\-ity theorems for all the other classical width parameters, such as tree- and path-width: for each of these we shall find an~$\F$, encoding some specific type of consistency, such that the graphs where this parameter is large are precisely those with an $\F$-tangle.%
 \COMMENT{}%
   \COMMENT{}
   The known duality theorems for these width parameters, in terms of concrete HCSs, will follow from our duality theorem in terms of abstract HCSs, but not conversely. This, too, has been shown in~\cite{TangleTreeGraphsMatroids}.

Our result will further imply duality theorems for $k$-blocks, the main concrete HCS for which no duality theorem has been known, and for any specified type of classical tangles (rather than all of them). This has been done in~\cite{ProfileDuality}.

Finally, our duality theorem has recently found an unexpected more fundamental application. The emerging theory of abstract separation systems and their tangles~-- see~\cite{duality1inf, SeparationsOfSets, CarmesinToTshort, carmesinhalinconj, confing, CanonicalTreesofTDs, EndsAndTangles, TreeSets, ProfileDuality, AbstractTangles, ProfiniteASS, TangleTreeGraphsMatroids, ToTfromTTD, FiniteSplinters, InfiniteSplinters, weighted_deciders_AIC, DirectedPDs, JoshRefining, JoshUnified, TreelikeSpaces, KneipProfiniteTreeSets, EndsAsTangles}~-- used to rest on two pillars: an abstract version~\cite{ProfilesNew} of Robertson and Seymour's tree-of-tangles theorem for graphs~\cite{GMX}, and the abstract tangle-tree duality theorem proved here. But, very recently, Elbracht, Kneip and Teegen~\cite{ToTfromTTD} have been able to derive the abstract tree-of-tangles theorem from the tangle-tree duality theorem. With this deduction, only one pillar remains: the result proved in this paper, and re-proved in~\cite{ToTfromTTD} in a slightly stronger form.

\medbreak

While the study of tangles as abstract HCSs marked a shift of paradigm from the earlier studies of concrete HCSs, there has since been another major shift of paradigm: from concrete to abstract {\it separations\/}. Separations in graphs~-- as well as traditional tangles and their dual branch decompositions~-- are defined in terms of the graph's edges. But when we proved our duality theorem for graphs we found that, surprisingly, we needed to know only how these separations relate to each other, not how they relate to the graph which they separate.

Our main result, therefore, is now a duality theorem for {\em abstract separation systems\/}. Very roughly, these are partially ordered sets (reflecting the natural partial ordering between separations in graphs and matroids), with an order-reversing involution that reflects the flip $(A,B)\mapsto (B,A)$ of a graph separation.

Both tangles in graphs and their dual tree structure can be expressed in terms of just this partial ordering of their separations. Indeed, the consistency requirement for tangles, that no three `small' sides of its oriented separations shall cover the graph, can be replaced by the requirement that whenever a tangle $\tau$ contains two oriented separations, $(A,B)$ and $(C,D)$ say, it also contains their supremum $(A\cup C, B\cap D)$ as long as this is oriented by $\tau$ at all, i.e., has order $<k$ if $\tau$ is a $k$-tangle. (Note the similarity to ultrafilters, a standard kind of abstract HCSs in infinite contexts.)
   And the \td s or branch-decompositions dual to graph tangles can be described purely in terms of the separations too, those that correspond to the edges of their decomposition trees, where the requirement that these edges form a tree can be replaced by requiring that those separations must be nested~-- which can in turn be expressed just in terms of our poset: two separations are {\em nested\/} if they have comparable orientations.

While our duality theorem for these abstract separation systems implies all the duality theorems mentioned so far, by applying it to separations in graphs or matroids, it can also be applied in very different contexts. These applications are surveyed in~\cite{TangleBookPreview,TanglesSocial,TangleBook}, in a style aimed at non-mathematicians in the sciences and in the quantitative social sciences, but also accessible to readers of this paper.\looseness=-1

As a generic such application outside mathematics consider cluster analysis. The bipartitions of  a (large data) set~$D$ form a separation system: they are partially ordered by inclusion of their sides, and the involution of flipping the sides of the bipartition inverts this ordering. Depending on the application, some ways of cutting the data set in two will be more natural than others, which gives rise to a cost function on these separations of~$D$.%
 \COMMENT{}
   Taking this cost of a separation as its `order' then gives rise to tangles:%
   \COMMENT{}
   abstract HCSs signifying clusters. Unlike clusters defined by simply specifying a subset of~$D$, clusters defined by tangles are allowed to be fuzzy (as in our earlier grid example)~-- which much enhances their real-world relevance. See~\cite{TangleClusteringWeakStrong} for more.

If the cost function on the separations of our data set is submodular~-- which in practice is not normally a severe restriction~-- our duality theorem can be applied to these tangles~\cite{TangleTreeGraphsMatroids}. For every integer~$k$, the application will either find a cluster of order at least~$k$ or produce a nested `tree' set of bipartitions, all of order~${<k}$, which together witness that no such cluster exists. An example from image analysis, with a cost function chosen so that the clusters become the visible regions in a picture, is given in~\cite{MonaLisa}.
   In an example from sociology, the yes/no questions of a political or social survey form a separation system whose tangles capture any existing {\it mindsets\/}: typical ways of answering its questions~\cite{TanglesSocial,TangleBook}. Tangles can identify such mindsets in a quantitative and precise way even if there does not exist any one set of complete answers that occurs more often than others. Our theorem, in addition, determines how dominant or prevalent such opinions are in the population surveyed, by finding the maximum $k$ for which there exist mindsets that define a $k$-tangle of answers.

There are also potential applications in pure mathematics. For a very simple example, consider a triangulation of a topological sphere. This can be cut in two, in many ways, by closed paths along the edges of the triangulation, i.e., by cycles in its 1-skeleton. The lengths of these paths or cycles%
   \COMMENT{}
   define a submodular order function on the separations of our sphere that they define. Our duality theorem then says that, for every integer~$k$, there either exists a region dense enough that no closed path of length~$<k$ can cut the sphere so as to divide this region roughly in half,%
   \COMMENT{}
   or there exists a collection of non-crossing paths each of length~$<k$ which, between them, cut up the entire sphere in a tree-like way (with a ternary tree)%
   \COMMENT{}
   into single triangles. If we do this in a geometric disc, where the edges in these triangulations have lenths, the inverse limits of these triangulations under refinement will have tangles describing visible `blobs', of varying granularity, of these geometric discs, and our theorem will tell us donw to which small granularity a given blob can be refined.%
   \COMMENT{}

Our abstract tangle-tree duality theorem will come in two flavours, `weak' and `strong'. Our {\it weak duality theorem\/}, presented in Section~\ref{sec:weak}, will be easy to prove but has no direct applications. It will be used as a stepping stone for the {\it strong duality theorem\/}, our main result, which we prove in Section~\ref{sec:strong}. In Section~\ref{sec:essence} we present a refinement of the strong duality theorem.

\section{Terminology and basic facts}\label{sec:def}

Our aim in this section is to introduce the reader to just enough of the terms and basic theory of abstract separation systems~\cite{ASS} to read this paper. In order to facilitate the build-up of enough intuition to flesh out the rather technical proof of our main theorem, however, we begin with a special case: some simple observations about separations of sets. The abstract definitions to follow will then be presented fairly concisely.

The ArXiv version~\cite{TangleTreeAbstractArXiv3} of this paper follows a slightly different approach here by seeking to motivate our abstract definitions before they are introduced, rather than just illustrating them afterwards by pointing out what they mean for separations of sets. Readers interested in why exactly our abstract notions are what they are, are invited to consults~\cite{TangleTreeAbstractArXiv3} for more insight.

\subsection{Separations of sets}

Let us start with a common generalization of separations in graphs on the one hand, and bipartitions of sets on the other. A~\emph{separation of a set} $V$ is a set $\{A,B\}$ such that ${A\cup B=V\!}$. Note that $A\cap B$ may be non-empty. It will usually be small, but technically also $\{V,V\} = \{V\}$ is allowed. The ordered pairs $(A,B)$ and $(B,A)$ are the two {\it orientations\/} of the (unoriented) separation~$\{A,B\}$ of~$V\!$.%
   \COMMENT{}
   (The separation $\{V,V\}$ has only the one orientation~$(V,V)$.) The {\em oriented separations\/} of~$V$ are the orientations of its separations: the ordered pairs $(A,B)$ such that ${A\cup B = V\!}$. Later, we shall use the term `separation' informally for either oriented or unoriented separations, as long as the context is clear.

   \begin{figure}[htpb]
   \centering
   	  \includegraphics{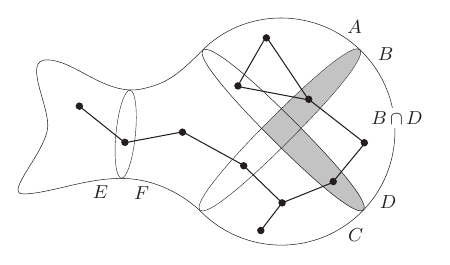}
   	        \vskip-6pt
\caption{Oriented separations $(E,F) < (A,B)$ and $(B,A) < (F,E)$ of a graph. The separations $(A,B)$ and~$(C,D)$ are incomparable.}
   \vskip-0pt\label{fishlemmafig}
   \end{figure}

Mapping every oriented separation $(A,B)$ to its {\it inverse\/} $(B,A)$ is an involution on the set of all oriented separations of~$V\!$. Note that it reverses the partial ordering of this set given by
 \[(A,B)\le (C,D) :\Leftrightarrow A\subseteq C \text{ and } B\supseteq D,\]
 since $(A,B)\le (C,D)$ is equivalent to $(B,A)\ge (D,C)$; see Figure~\ref{fishlemmafig}. Informally, we think of $(A,B)$ as \emph{pointing towards}~$B$ and \emph{away from}~$A$.%
   \COMMENT{}

\subsection{Separation systems}

Generalizing these properties of separations of sets, we now give an axio\-matic definition of {\em abstract separations\/}: these do not have to `separate' anything, but their properties will resemble the properties of separations of sets we just outlined above.%
 \COMMENT{}
   We shall define oriented (abstract) separations first, and then pair them with their inverses to form unoriented separations as quotients of oriented ones.

A~{\em separation system\/} $(\vS,\le\,,\!{}^*)$ is a partially ordered set $\vS$ with an order-reversing involution~*. Its elements are called {\em oriented separations\/}. When a given element of $\vS$ is denoted as~$\vs$, its {\em inverse\/}~$\vs{}^*$ will be denoted as~$\sv$, and vice versa. The assumption that * be {\em order-reversing\/} means that, for all $\vr,\vs\in\vS$,
\begin{equation}\label{invcomp}
\vr\le\vs\ \Leftrightarrow\ \rv\ge\sv.
\end{equation}
An (unoriented) {\em separation\/} is a set of the form $\{\vs,\sv\}$, and then denoted by~$s$. We call $\vs$ and~$\sv$ the {\em orientations\/} of~$s$. We say that $\vr$ {\em points towards\/}~$s$, and $\rv$ {\em points away from\/}~$s$, if $\vr\le\vs$ or $\vr\le\sv$. In Figure~\ref{fishlemmafig}, for example, the oriented separation $(E,F)$ points towards the separations $\{A,B\}$ and~$\{C,D\}$.

The set of all unoriented separations $s=\{\vs,\sv\}\sub\vS$ will be denoted by~$S$. If $\vs=\sv$, we call both $\vs$ and $s$ {\em degenerate\/}. A~set~$V\!$, clearly, has exactly one degenerate oriented separation: the separation~$(V,V)$.

When a separation is introduced notationally ahead of its elements, and denoted by a single letter~$s$, say, then its elements will subsequently be denoted as $\vs$ and~$\sv$.%
 \COMMENT{}
   Given a set $S'\sub S$ of unoriented separations, we write $\vSdash := \bigcup S'\sub\vS$%
   \COMMENT{}
   for the set of all the orientations of its elements. With the ordering and involution induced from~$\vS$, this is again a separation system.%
 \COMMENT{}

Separations of sets, and their orientations, are clearly an instance of this if we identify $\{A,B\}$ with $\{(A,B),(B,A)\}$.

If there are binary operations $\vee$ and~$\wedge$ on our separation system~$\vS$ such that $\vr\vee\vs$ is the supremum and $\vr\wedge\vs$ the infimum of $\vr$ and~$\vs$ in~$\vS$, we call $(\vS,\le\,,\!{}^*,\vee,\wedge)$ a {\em universe\/} of (oriented) separations. By~\eqref{invcomp}, it satisfies De~Mor\-gan's law:
\begin{equation}\label{deMorgan}
   (\vr\vee\vs)^* =\> \rv\wedge\sv.
\end{equation}%
   \COMMENT{}

The oriented separations of a set~$V$ form such a universe: if $\vr = (A,B)$ and $\vs = (C,D)$, say, then $\vr\vee\vs := (A\cup C, B\cap D)$ and $\vr\wedge\vs := (A\cap C, B\cup D)$ are again oriented separations of~$V$, and are the supremum and infimum of $\vr$ and~$\vs$, respectively. Similarly, the oriented separations of a graph (see~\cite{DiestelBook16}) form a universe of separations (see Figure~\ref{fishlemmafig}). Its oriented separations of order~$<k$ for some fixed~$k$, however, form a separation system~$\vS_k$ inside this universe that may not itself be a universe (with the same definition of $\vee$~and~$\wedge$). This is because the separations $(A\cup C, B\cap D)$ and $(A\cap C, B\cup D)$ may have an order greater than~$k$, and then fail to lie in~$\vS_k$.%
   \COMMENT{}

\subsection{Small and trivial separations}

A separation $\vr\in\vS$ is {\it small\/} if $\vr\le\rv$. The set of small separations is closed down in~$\vS$: if $\vs$ is small then so is any $\vr\le\vs$, because $\vr\le\vs\le\sv\le\rv$ by~\eqref{invcomp}. The small separations of a set~$V\!$ are those of the form~$(A,V)$.

A separation $\vr\in\vS$ is {\it trivial in~$\vS$}, and its inverse~$\rv$ {\em co-trivial\/}, if there exists $s \in S$ such that $\vr < \vs$ as well as $\vr < \sv$.%
 \COMMENT{}
    Such an $s$ is a {\em witness\/} of $\vr$ and its triviality. The trivial separations of a set~$V\!$, in the system~$\vS$ of all its separations, are those of the form $\vr = (X,V)$ for which there exists $s = \{A,B\}\in S\sm\{r\}$ with $X\sub A\cap B$.

All trivial separations are small: if $s$ witnesses the triviality of~$\vr$, then $\vr < \vs$ as well as  $\vr < \sv$, and hence $\vr < \vs < \rv$ by~\eqref{invcomp}. As these inequalities are strict, trivial separations are never degenerate.


Small but nontrivial separations can exist but are rare: only the maximal small separations in~$\vS$ can be nontrivial. Indeed, if $\vs$ is small then every $\vr < \vs$ is not only small but in fact trivial, since $\vr < \vs\le\sv$. As trivial separations are small, this means that they, too, are closed down in~$\vS$. We thus have two down-closed subsets of~$\vS$: the set of trivial separations, and the (possibly) slightly larger set of small separations.%
   \COMMENT{}

\subsection{Nestedness and consistency}

Two separations $r,s\in S$ are {\em nested\/} if they have comparable orientations; otherwise they \emph{cross}. Two oriented separations $\vr,\vs$ are {\em nested\/} if $r$ and~$s$ are nested.%
   \footnote{Terms introduced for unoriented separations may be used informally for oriented separations too if the meaning is obvious, and vice versa.}%
   \COMMENT{}
   Thus, two nested oriented separations are either comparable, or point towards each other, or point away from each other. A~set of separations is {\em nested\/} if every two of its elements are nested. In Figure~\ref{fishlemmafig}, the separations $\{A,B\}$ and~$\{C,D\}$ cross but are nested with~$\{E,F\}$.

A set $O\sub \vS$ of oriented separations is {\em antisymmetric\/} if it does not contain the inverse of any of its nondegenerate elements.%
   \COMMENT{}
   It is \emph{consistent} if there are no distinct $r,s\in S$ with orientations $\vr < \vs$ such that $\rv,\vs\in O$.%
   \COMMENT{}
   In other words, a set of oriented separations is {\em consistent\/} if no two of its elements that are orientations of distinct separations point away from each other.

An \emph{orientation} of a set~${S}$ of separations is a set $O\sub{\vS}$ that contains for every $s\in{S}$ exactly one of its orientations $\vs,\sv$. A \emph{partial orientation} of~${S}$ is an orientation of a subset of~${S}$, i.e., an antisymmetric subset of~$\vS$.

Every consistent orientation of $S$ contains all separations $\vr$ that are trivial in~$\vS$, because it cannot contain their inverse~$\rv$: if the triviality of $\vr$ is witnessed by $s\in S$, say, then $\rv$ would be inconsistent with both $\vs$ and~$\sv$. It is not hard to show that every consistent partial orientation of $S$ containing no co-trivial $\rv\in\vS$ extends to a consistent orientation of all of~$S$; see~\cite{TreeSets}.

\subsection{Stars of separations}

A~family $(\,\vsi \mid i\in I\,)$, possibly empty, of nondegenerate oriented separations is a \emph{multistar of separations} if they point towards each other, that is, if $\vsi\le\svj$ for all distinct $i,j\in I$ (Figure~\ref{fig:star}). Note that if a multistar contains a separation $\vs$ more than once then $\vs$ must be small.

   \begin{figure}[htpb]
   \centering
   	  \includegraphics{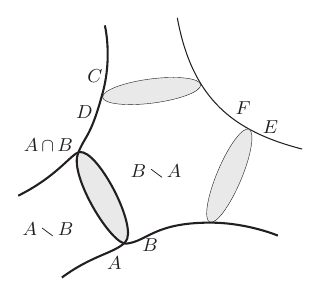}
      \vskip-9pt
   	  \caption{The separations $(A,B),(C,D),(E,F)$ form a star.}
   \vskip-3pt\label{fig:star}
   \end{figure}

Multistars in which each element occurs only once are called {\em stars\/}. To avoid notational hairsplitting, we think of stars as the obvious sets rather than as families, and say that a multistar $(\,\vsi \mid i\in I\,)$ {\em induces\/} the star $\{\,\vsi \mid i\in I\,\}$ obtained from it by forgetting the multiplicities of its elements.

Multistars of separations are clearly nested. They are also consistent: if $\rv,\vs$ lie in the same multistar we cannot have $\vr < \vs$, since also $\vs\le \vr$ by the multistar property.

Note that a multistar~$\sigma$ need not, by definition, be antisymmetric.%
   \COMMENT{}
   But if it is not, i.e.\ if $\{\vs,\sv\}\sub\sigma$, say, then any other $\vr\in\sigma$ will be trivial, witnessed by~$s$. Hence most of the stars we have to deal with will in fact be antisymmetric, but it is important to keep this example in mind as a pathological case that can, and will, occur.

Given a tree~$T$ (which, by definition~\cite{DiestelBook16}, has at least one node),%
   \COMMENT{}
  there is a \emph{natural partial ordering} on the set
    $$\vec E(T) := \{\, (x,y) : \{x,y\}\in E(T)\,\}$$
 of its oriented edges defined by letting $(x,y) < (u,v)$ if $\{x,y\}\ne \{u,v\}$ and the unique $\{x,y\}$--$\{u,v\}$ path in $T$ joins $y$ to~$u$ (see Figure~\ref{Stree}). For each node $t$ of~$T$, we call the set
 $$\vec F_t := \{(x,t) : xt\in E(T)\}$$
 of its incoming oriented edges the {\em oriented star at~$t$} in~$T$. This is also a star in the separation system $(\vec E(T),\le,{}^*)$, where $\le$ comes from the natural partial ordering on~$\vec E(T)$ and ${}^*\colon (x,y)\mapsto (y,x)$ flips the orientations of edges.

\subsection{\boldmath $S$-trees}

Let $(\vS,\le\,,\!{}^*)$ be a separation system.  An \emph{${S}$-tree\/} is a pair $(T,\alpha)$ of a tree~$T$ and a function $\alpha\colon\vec E(T)\to \vS$ that commutes with the involutions, i.e.,\ satisfies $\alpha(\ev) = \alpha(\ve)^*$ for all $\ve\in\vec E(T)$. If $T$ has an edge and we consider it as rooted at a leaf~$x$, then this implicitly defines its oriented edge emanating from~$x$ as~$\vex$.

For every node~$t\in T$, the families $(\,\alpha(\ve)\mid\ve\in\vec F_t\,)$%
   \COMMENT{}
   and the sets $\alpha(\vec F_t)$ of separations in~$\vS$ are said to be {\em associated with}~$t$ in~$(T,\alpha)$. If all the sets $\alpha(\vec F_t)$ are elements of some set~$\F$%
   \COMMENT{}
   we say that $(T,\alpha)$ is an $S$-tree {\em over}~$\F$. If the elements of $\F$ are stars, we also say that $(T,\alpha)$ is an $S$-tree {\em over stars.}

If $\alpha$ preserves the natural ordering on $\vec E(T)$, i.e., if for all $\ve,\vf\in\vec E(T)$ with $\ve\le\vf$ we have $\alpha(\ve)\le\alpha(\vf)$ in~$\vS$ (Fig.~\ref{Stree}), we call $(T,\alpha)$ {\em order-respecting\/}.

   \begin{figure}[htpb]
\centering
   	  \includegraphics{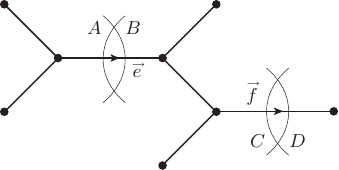}
   	  \caption{If $(T,\alpha)$ is order-respecting, then $\ve < \vf$ implies $(A,B)\le (C,D)$.}
   \label{fig:Stree}\vskip-6pt\label{Stree}
   \end{figure}

Note that the map $\alpha$ in an order-respecting $S$-tree $(T,\alpha)$ need not preserve strict inequalities: it can happen for $\ve < \vf$ that $\alpha(\ve) = \alpha(\vf)$. Similarly, while the sets $\vec F_t\sub\vec E(T)$ are stars, by definition, their images $(\,\alpha(\ve)\mid\ve\in\vec F_t\,)$ as families will, in general, only be multistars. The sets $\alpha(\vec F_t)\sub\vS$, then, are the stars in~$\vS$ which these multistars induce.

While order-respecting $S$-trees are clearly over stars, $S$-trees over stars need not be order-respecting. For example, let $T_3$ be obtained from the 3-star with centre~$t$ and leaves $x_1,x_2,x_3$ by subdividing each edge $x_i t$ by a new vertex~$y_i$. Let $\alpha_3$ map all the oriented edges $(y_i, t)$ to the same%
   \COMMENT{}
   separation~$\vs$, and the edges $(x_i,y_i)$ to separations~$\vri<\vs$, all nondegenerate. Then~$(T_3,\alpha_3)$ is an $S$-tree over stars. In particular, $\alpha_3(\vec F_t) = \{\vs\}$ is a star, even though $(\alpha_3(y_1,t),\alpha_3(y_2,t),\alpha_3(y_3,t))$ is not a multistar (unless $\vs$ is small). So the separations $r_i$ need not be nested with each other: it is easy to think of examples where three separations $\vri < \vs$ ($i=1,2,3$) cross pairwise. And if the $\vri$ are not nested, then $(T_3,\alpha_3)$ will not be order-respecting.%
   \COMMENT{}

However, this example is essentially the only one. Indeed, let us call an $S$-tree $(T,\alpha)$ {\em redundant\/} if it has a node $t$ of~$T$ with distinct neighbours $t',t''$ such that $\alpha(t',t) = \alpha(t'',t)$; otherwise we call it {\em irredundant\/}. Figure~\ref{tdfig} shows an irredundant $S$-tree over stars, so $\alpha(\vec E(T))\sub\vS$ is nested.

   \begin{figure}[htpb]
\centering
   	  \includegraphics{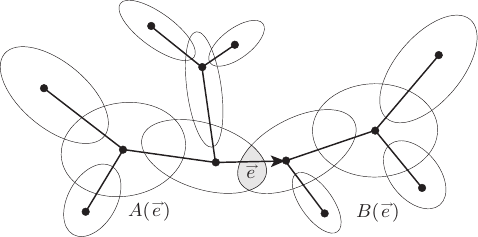}
   	  \caption{An irredundant $S$-tree $(T,\alpha)$ over stars in a graph. The ovals indicate the sets $\bigcap_{\ve\in\vec F_t} B(\ve)$ with $\alpha(\ve) =: (A(\ve),B(\ve))$, one for each node $t$ of~$T$.}
   \label{fig:Stree}\vskip-6pt\label{tdfig}
   \end{figure}

\begin{LEM}\label{preservele}{\rm \cite{TreeSets}}\COMMENT{}
Every irredundant $S$-tree $(T,\alpha)$ over stars is order-respecting. In particular, $\alpha(\vec E(T))$ is a nested set of separations in~$\vS$.
\end{LEM}

\begin{proof}
As $(T,\alpha)$ is irredundant, our assumption that the sets $\alpha(\vec F_t)$ are stars is tantamount to saying that the families $(\alpha(\ve)\mid \ve\in\vec F_t)$ are multistars. In other words, the $S$-trees which $(T,\alpha)$ induces on its maximal stars, the subtrees consisting of a fixed node~$t$ and the neighbours of~$t$, are order-respecting. As the relation~$\le$ is transitive, this propagates through~$\vec E(T)$ to make the entire $(T,\alpha)$ order-respecting.
\end{proof}

Note that Lemma~\ref{preservele} does not have a direct converse: order-respecting $S$-trees over stars can be redundant. Indeed, if the separation $\vs$ in our earlier example of~$(T_3,\alpha_3)$ is small, i.e., satisfies $\vs\le\sv$, then $(T_3,\alpha_3)$ is an order-respecting but redundant $S$-tree over stars.%
 \COMMENT{}%
   \COMMENT{}
   See~\cite[Lemma 6.3]{TreeSets} for more details on how the orderings on $\vec E(T)$ and its $\alpha$-image in~$\vS$ are related or not.

Two edges of an irredundant $S$-tree over stars cannot have orientations that point towards each other and map to the same separation, unless this is trivial:

\begin{LEM}\label{NewLemma}{\rm \cite{TreeSets}}\COMMENT{}
Let $(T,\alpha)$ be an irredundant%
   \COMMENT{}
   $S$-tree over a set $\F$ of stars. Let $e,f$ be distinct edges of~$T$ with orientations $\ve < \vf$ such that $\alpha(\ve) = \alpha(\fv) =:\vr$. Then $\vr$ is trivial.

In particular, $T$~cannot have distinct leaves associated with the same star~$\{\rv\}$ unless $\vr$ is trivial.%
   \COMMENT{}
\end{LEM}%
   \COMMENT{}

Note that the irredundancy assumption in Lemma~\ref{NewLemma} cannot be replaced by the weaker%
   \COMMENT{}
   assumption that $(T,\alpha)$ be order-respecting: if $T$ is a 3-star whose edges oriented towards the centre map to the same small but nontrivial separation$\,\vs$, then $(T,\alpha)$ is order-respecting but its three leaves are associated with the same star~$\{\sv\}$ with $\vs$ nontrivial.

Redundant $S$-trees can clearly be {\em pruned\/} to irredundant ones over the same~$\F$, by deleting entire branches that hang off edges causing a redundancy:

\begin{LEM}\label{prune}{\rm \cite{TreeSets}}\COMMENT{}
If $(T,\alpha)$ is an $S$-tree over~$\F$, possibly redundant, then $T$ has a subtree~$T'$ such that $(T',\alpha')$ is an irredundant $S$-tree over~$\F$, where $\alpha'$ is the restriction of $\alpha$ to $\vec E(T')$. If $(T,\alpha)$ is rooted at a leaf~$x$ and $T$ has an edge, then $T'$ can be chosen so as to contain $x$ and~$e_x$.\qed
\end{LEM}%
   \COMMENT{}

Recall that stars of separations need not, by definition, be antisymmetric. While it is important for our proofs to allow this,%
   \COMMENT{}
   we can always contract an $S$-tree $(T,\alpha)$ over a set $\F$ of stars to an $S$-tree $(T',\alpha')$ over the subset $\F'\sub\F$ of its antisymmetric stars. Indeed, if $T$ has a node~$t$ such that $\alpha(\vec F_t)$ is not antisymmetric, then $t$ has neighbours $t',t''$ such that $\alpha(t',t) = \vs = \alpha(t,t'')$ for some $\vs\in \vS$. Let $T'$ be the tree obtained from~$T$ by deleting the component of $T-t't - tt''$ containing~$t$ and joining $t'$ to~$t''$. Let $\alpha'(t',t'') := \vs$ and $\alpha'(t'',t') := \sv$, and otherwise let $\alpha':= \alpha\!\restriction\! \vec E(T')$. Then $(T',\alpha')$ is again an $S$-tree over~$\F$. Since we can do this whenever some $\vec F_t$ maps to a star of separations that is not antisymmetric, but only finitely often, we must arrive at an $S$-tree over~$\F'$.

An $S$-tree $(T,\alpha)$ is called {\em tight\/} if all the sets $\alpha(\vec F_t)$ for nodes $t\in T$ are antisymmetric. The reduction described above thus turns an arbitrary $S$-tree over~$\F$ into a tight $S$-tree over~$\F'$.

\begin{LEM}\label{onlyedge}{\rm \cite{TreeSets}}%
   \COMMENT{}
   Let $(T,\alpha)$ be an $S$-tree over a set $\F$ of stars, rooted at a leaf~$x$. Assume that $T$ has an edge, and that $\vr = \alpha(\vex)$ is nontrivial.%
   \COMMENT{}
   Then $T$ has a minor $T'\!$ containing~$x$ and~$e_x$ such that $(T',\alpha')$, where $\alpha' = \alpha\!\restriction\!\vec E(T')$, is a tight and irredundant $S$-tree over~$\F$.

For every such $(T',\alpha')$ the edge $\vex$ is the only edge $\ve\in\vec E(T')$ with $\alpha(\ve)=\vr$.%
   \COMMENT{}
\end{LEM}

\begin{proof}
By Lemma~\ref{prune} we may assume that $(T,\alpha)$ is irredundant and contains both~$x$ and~$e_x$. We now apply to $(T,\alpha)$ the reduction described before this lemma to obtain a tight and irredundant $S$-tree $(T',\alpha')$ over~$\F$.

Let us show that $T'$ still contains $x$ and~$e_x$. When, in the said reduction process, we formally deleted an edge~$e$ at a node~$t$%
   \COMMENT{}
   (with $\ve$ oriented towards~$t$, say), then $\alpha(\ve)$ was trivial, witnessed by~$s$. (As $(T,\alpha)$ is irredundant, we have $\alpha(\ve)\notin\{\vs,\sv\}$.) But every other edge $e'$ deleted at that step had an orientation $\vedash<\ve$ for such an edge~$e$ at~$t$, making $\alpha(\vedash)$ trivial too. As $(T,\alpha)$ was order-respecting (Lemma~\ref{preservele}), we thus never deleted~$e_x$,  because $\vr = \alpha(\vex)$ was nontrivial by assumption.

It remains to show that $\vex$ is the only edge $\ve\in\vec E(T')$ with $\alpha(\ve)=\vr$. If there is another such edge~$\ve$, then $e\ne e_x$, since otherwise $\ve=\evx$ and hence $\alpha(\ve) = \rv$ as well as $\alpha(\ve) = \vr$, which would make $r$ degenerate, contradicting the fact that $(T,\alpha)$ is an $S$-tree over stars.

By Lemma~\ref{NewLemma} we cannot have $\vex < \ev$,%
   \COMMENT{}
   so $\vex < \ve$ since $\vex$ issues from a leaf. By Lemma~\ref{preservele}, every edge $\vedash$ with $\vex\le\vedash\le\ve$ satisfies $\vr = \alpha(\vex)\le\alpha(\vedash)\le\alpha(\ve) = \vr$, so $\alpha(\vedash) = \vr$. Our assumption of $\ve\ne\vex$ thus implies that $(T',\alpha')$ is not tight, a contradiction.
  \end{proof}

\section{Weak duality}\label{sec:weak}

Our paradigm in this paper is to capture the notion of `highly cohesive substructures' in a given  combinatorial structure by orientations $O$ of a set $S$ of separations of this structure that satisfy certain consistency rules laid down by specifying a set $\F$ of `forbidden' sets of oriented separations that $O$ must not contain.\looseness=-1

Let us say that a partial orientation $P$ of $S$ \emph{avoids} $\F\sub 2^{\vS}$ if $2^P\cap\F = \emptyset$.

\begin{THM}[Weak Duality Theorem]\label{thm:weak}
  Let $(\vS,\le\,,\!{}^*)$ be a finite separation system and $\F\sub 2^\vS\!$ a set of stars.%
   \COMMENT{}
   Then exactly one of the following assertions holds:
\begin{enumerate}[\rm(i)]\itemsep0pt
  \item There exists an ${S}$-tree over~$\F$.%
    \COMMENT{}
  \item There exists an orientation of~${S}$ that avoids~$\F$.%
      \COMMENT{}
  \end{enumerate}
\end{THM}

\noindent
We remark that, by Lemma~\ref{prune}, the $S$-tree in~(i) can be chosen irredundant, in which case it will be order-respecting by Lemma~\ref{preservele}.

For our proof of Theorem~\ref{thm:weak} we need the following simple lemma, whose proof uses the fact that every orientation of a finite tree has a sink. To find one, just follow a maximal directed path.

\begin{LEM}\label{lem:notboth}
  Let $(\vS,\le\,,\!{}^*)$ be a separation system and $\F\sub 2^\vS$.
  If there exists an ${S}$-tree over~$\F$, then no orientation of~${S}$ avoids~$\F$.
\end{LEM}

\begin{proof}
  Let $(T,\alpha)$ be an ${S}$-tree over $\F$, and let $O$ be an orientation of~${S}$. Let $t\in V(T)$ be a sink in the orientation of the edges of~$T$ that $O$ induces via~$\alpha$.%
   \COMMENT{}
   Then $\alpha(\vec F_t) \sub O$. Since $\alpha(\vec F_t)\in\F$, as $(T,\alpha)$ is an $S$-tree over~$\F$, this means that $O$ does not avoid~$\F$.
\end{proof}

Before we launch into the proof of Theorem~\ref{thm:weak}, let us sketch its idea. Our aim will be to find either an $\F$-avoiding orientation~$O$ of~$S$ or construct an $S$-tree over~$\F$ to witness that no such orientation exists (by Lemma~\ref{lem:notboth}). If $\F$ contains any singleton sets~$\{\sv\}$, then every $\F$-avoiding orientation of~$S$ must contain~$\vs$ rather than~$\sv$. We think of these $\vs$ as `forced' by~$\F$, and will apply induction on the number of separations in~$S$ neither of whose orientations is forced.

In the induction step, we shall consider some such~$s$, call it~$s_0$, and see what happens if we force one of its orientations by adding either $\{\svo\}$ or~$\{\vso\}$ to~$\F$. Then the induction will give us either an orientation of~$S$ that avoids one of these augmented~$\F$, and hence also the original~$\F$, or two $S$-trees, one over each augmented~$\F$. If one of these is an $S$-tree even over the original~$\F$, we are again done, so we assume not.

Then one of these $S$-trees contains a leaf associated with~$\{\svo\}$, while the other contains a leaf associated with~$\{\vso\}$. Assume, for simplicity, that these are the only such leaves in their respective $S$-trees. We can then combine these two trees into a single $S$-tree over our original~$\F$ by identifying those two leaves and then suppressing the identified node, completing the proof.

It will help, also with the more difficult proof of our strong duality theorem in Section~\ref{sec:strong}, to visualize the outline above once more for the case when $S$ consists of separations of a graph~$G$. Then $s_0$ will be a separation~$\{A,B\}$ of~$G$, and the two $S$-trees we obtain from the induction hypothesis will essentially be $S$-trees over~$\F$ of the two sides of this separation, of the graphs~$G[A]$ and~$G[B]$. Only `essentially', because they will each have one additional leaf, associated with $(A,B)$ or~$(B,A)$, respectively. In the \td s naturally associated with these $S$-trees, these leaf nodes will correspond to the bag~$B$ or the bag~$A$, respectively. The rest of these $S$-trees will decompose the other side of $s_0 = \{A,B\}$, the graph $G[A]$ or~$G[B]$.

\begin{proof}[Proof of Theorem~\ref{thm:weak}]
  By Lemma~\ref{lem:notboth}, at most one of (i) and (ii) holds.
  We now show that at least one of them holds. Let
 $$O^- := \{\,\vs\mid \{\sv\}\in\F\,\}.$$
    Then any $\F$-avoiding orientation of~$S$ must include~$O^-$ as a subset. As $\F$ consists of stars, $O^-$~contains no degenerate separations.

If $O^-\supe \{\vs,\sv\}$ for some~$s\in S$, then $(T,\alpha)$ with $T=K_2$ and ${\rm im}\,\alpha=\{\vs,\sv\}$ is an $S$-tree over~$\F$. So we may assume that $O^-$ is antisymmetric: a~partial orientation of~$S\sm D$, where $D$ is the set of degenerate elements of~$S$. We apply induction on $|S\sm D|-|O^-|$ to show that, whenever $\F$ is such that $O^-$ is antisymmetric, either (i) or (ii) holds.

  If $|S\sm D| = |O^-|$, then ${O^-\cup\vec D}$ is an orientation of all of~${S}$. If (ii) fails then ${O^-\cup\vec D}$~has a subset $\sigma\in\F$. As $\F$ consists of stars we have $\sigma\cap\vec D = \emptyset$, so $\sigma\sub O^-$. By definition of~$O^-$, and since $O^-$ is antisymmetric, $\sigma$~is not a singleton set%
   \COMMENT{}
   (though it may be empty).%
   \COMMENT{}
   Let $T$ be a star of $|\sigma|$ edges with centre~$t$, say, and let $\alpha$ map its oriented edges $(x,t)$ bijectively to~$\sigma$. Then $(T,\alpha)$ satisfies~(i).%
   \COMMENT{}

We may thus assume that ${S\sm D}$ contains a separation $s_0$ such that neither $\vso$ nor~$\svo{}$ is in ${O^-}$. Let
 $$\F_1 := \F\cup\{\{\vso\}\}\quad\text{and}\quad\F_2 := \F\cup\{\{\svo{}\}\},$$
 and put $O^-_i  := \{\,\vs\mid \{\sv\}\in\F_i\,\}$ for $i=1,2$. Note that $|O^-_i| > |O^-|$, and $O^-_i$ is again antisymmetric.%
  \COMMENT{}

Since any $\F_i$-avoiding orientation of~$S$ also avoids~$\F$, we may assume for both $i=1,2$ that no orientation of~$S$ avoids~$\F_i$. By the induction hypothesis, there are ${S}$-trees $(T_i,\alpha_i)$ over~$\F_i$. Unless one of these is in fact over~$\F$,%
   \COMMENT{}
   the tree $T_1$~has a leaf~$x_1$ associated with~$\{\vso\}$, while $T_2$ has a leaf $x_2$ associated with~$\{\svo\}$. Use Lemma~\ref{prune} to prune the $(T_i,\alpha_i)$ to irredundant $S$-trees $(T'_i,\alpha'_i)$ over~$\F_i$ containing $x_i$ and~$e_{x_i}$. Suppose first that $s_0$ has no trivial orientation. Then Lemma~\ref{NewLemma} implies that $x_1$ and $x_2$ are the only leaves of $T'_1$ and~$T'_2$ associated with $\{\vso\}$ and~$\{\svo\}$, respectively.

  Let $T$ be the tree obtained from the disjoint union of~$T'_1-x_1$ and $T'_2-x_2$ by joining  the neighbour $y_1$ of $x_1$ in~$T'_1$ to the neighbour~$y_2$ of~$x_2$ in~$T'_2$. Let $\alpha\colon\vec E(T)\to \vS$ map $(y_1,y_2)$ to~$\vso$ and $(y_2,y_1)$ to~$\svo{}$ and otherwise extend $\alpha'_1$ and~$\alpha'_2$. Then $\alpha'_1(y_1,x_1) = \alpha(y_1,y_2) = \alpha'_2(x_2,y_2)$, so $\alpha$ maps the oriented stars of edges at $y_1$ and~$y_2$ to the same multistars of separations in~$\vS$ as~$\alpha'_1$ and~$\alpha'_2$ did. The stars they induce lie in~$\F$, so $(T,\alpha)$ is an ${S}$-tree over~$\F$.

Suppose now that $\vso$, say, is trivial. Then $\svo$ is nontrivial, and $x_1$ is the only leaf of~$T'_1$ associated with~$\{\vso\}$, by Lemma~\ref{NewLemma} as before. Let $x_2^1,\dots,x_2^n$ be the leaves of~$T'_2$ associated with~$\{\svo\}$,%
   \COMMENT{}
   and let $T$ be obtained from the union of $T'_2 - \{x_2^1,\dots,x_2^n\}$ with $n$ copies of~$T'_1 - x_1$ by joining, for all $i=1,\dots,n$, the neighbour $y_1^i$ of~$x_1$ in the $i$th copy of~$T'_1 - x_1$ to the neighbour $y_2^i$ of~$x_2^i$ in~$T'_2$. Define $\alpha\colon \vec E (T)\to\vS$ as earlier, mapping $(y_1^i,y_2^i)$ to~$\vso$ and $(y_2^i,y_1^i)$ to~$\svo$ for all~$i$, and otherwise extending $\alpha'_1$ and~$\alpha'_2$.
   \end{proof}

\section{Strong duality}\label{sec:strong}

Theorem~\ref{thm:weak}, alas, has a serious shortcoming: there are few, if any, sets $S$ and $\F\sub 2^\vS$ such that $\F$ consists of stars in~$\vS$ and the $\F$-avoiding orientations of~$S$ (all of them) capture an interesting notion of highly cohesive substructure found in the wild. The reason for this is that we are not, so far, requiring these orientations $O$ to be consistent: we allow that $O$ contains separations $\rv$ and $\vs$ when $\vr < \vs$, which will not usually be the case when $O$ is induced by a meaningful highly cohesive substructure in the way discussed earlier. (We cannot simply add such sets $\{\rv,\vs\}$ to~$\F$, since in order to be able to use Lemma~\ref{NewLemma} we must assume that $\F$ consists of stars of separations.)

So what happens if we strengthen (ii) so as to ask for a consistent orientation of~$S$? Let us call a consistent $\F$-avoiding orientation of~$S$ an {\em $\F$-tangle\/}. Since all consistent orientations of~$S$ will contain all trivial $\vr\in\vS$, we may then add all co-trivial singletons~$\{\rv\}$ to~$\F$ without impeding the existence of an $\F$-tangle; this might help us find an $S$-tree over~$\F$ if no such orientation exists.

Still, our proof breaks down as early as the induction start: we now also have to ask that $O^-$~-- indeed, $O^-\cup \vec D$~-- should be consistent. It is not even unnatural to ensure that $O^-$ is closed down in~$(\vS,\le)$ (which implies consistency),%
   \COMMENT{}
   by requiring that if $\{\vr\}\in\F$ and $\vr < \vs$ then also $\{\vs\}\in\F$. For if a singleton star $\{\vr\}$ is in~$\F$, the idea is that the part of our structure to which $\vr$ points is too small to contain a highly cohesive substructure; and then the same should apply to all $\vs$ with $\vr < \vs$.%
   \COMMENT{}

But now we have a problem at the induction step: when forming the~$\F_i$, we now have to add not only $\{\vso\}$ or~$\{\svo{}\}$ to~$\F$, but all singleton stars $\{\vs\}$ with $\vso\le\vs$ or $\svo{}\le\vs$, respectively, to keep the $O_i^-$ closed down. This, then, spawns more problems: now both $T_i$ can have many leaves associated with a singleton star of $\F_i$ that is not in~$\F$.%
   \COMMENT{}
   Even if each of these occurs at most once, there is no longer an obvious way of how to merge $T_1$ and~$T_2$ into a single $S$-tree over~$\F$.

We shall deal with this problem as follows. Rather than adding singletons of the orientations of some fixed separation $s_0$ to~$\F$ to form the~$\F_i$ for $i=1,2$, we shall provisionally add, separately for $i=1,2$, some~$\{\rvi\}$ such that $\vri$ is minimal in~$\vS\sm(O^-\cup\vec D)$. This will most likely mean that $r_1\ne r_2$. But $\{\rvi\}$ can be associated with at most one leaf of~$(T'_i,\alpha'_i)$, because $\vri\notin O^-$ will still%
   \COMMENT{}
   be nontrivial (cf.\ Lemma~\ref{NewLemma}). 

If $r_1\ne r_2$, however, we shall no longer be able to combine the two $S$-trees over $\F\cup\{\rvone\}$ and $\F\cup\{\rvtwo\}$ into a single $S$-tree over~$\F$, as we did in the proof of the weak duality theorem: this step hinged on the fact that the two oriented separations whose singleton sets we added to~$\F$ were orientations of the same separation~$s_0$, which enabled us to think of each of these $S$-trees as decomposing one `side' of~$s_0$. (We illustrated this for the case of graph separations, where for $s_0=\{A,B\}$ the two $S$-trees decomposed the two sides $A$ and $B$ of the graph separately and could thus be joined into a single decomposition of the entire graph.) However, we shall be able to adapt that idea as follows.

Our two separations $\vri$ will be chosen nested, so that both point to some $s_0$ between them. We shall then modify the two $S$-trees over the~$\F_i = \F\cup\{\rvi\}$ into $S$-trees over $\F\cup\{\{\vso\}\}$ and~$\F\cup\{\{\svo\}\}$, respectively, by `shifting' the separations to which they map their edges to either side of~$s_0$, and then merge these shifted $S$-trees as before to obtain one over~$\F$.

To illustrate this, let us again consider the case that $S$ consists of separations of a graph, with $s_0 = \{A,B\}$ say. We shall have to turn the separations of~$G$ to which the first $S$-tree maps its edges into separations essentially of~$G[A]$ (though formally still of~$G$), those for the other $S$-tree into separations essentially of~$G[B]$ (but formally still of~$G$). The modified separations decomposing~$G[A]$ will be nested since they came from the first $S$-tree, and will in addition be nested with~$\{A,B\}$ (by definition: this is the result of shifting them). Likewise, the modified separations decomposing~$G[B]$ will be nested with each other and also with $\{A,B\}$. And the separations from the first collection cannot cross those from the second because they `lie on' different sides of~$\{A,B\}$. (In abstract terms: $\svo$ will point to the first lot, and $\vso$ to the second.) So the union of these two nested sets of separations, one from each of the two $S$-trees provided by the induction hypothesis, will also be nested, and in addition nested with~$\{A,B\}$. The two $S$-trees from which they come can thus be merged along the (unique) edge mapping to~$\{A,B\}$ into a single $S$-tree over~$\F$.

\medbreak

Let us now define this shifting operation. Consider a separation system $(\vS,\le\,,\!{}^*)$ contained in some universe of separations $(\vU,\le\,,\!{}^*,\vee,\wedge)$,%
   \COMMENT{}
   the ordering and involution on~$\vS$ being induced by those of~$\vU$. Let $(T,\alpha)$ be an $S$-tree rooted at a leaf~$x$, and let $\vso\in\vS$ be nondegenerate and such that $\alpha(\vex)\le\vso$. Then let $\alpha' = \alpha_{x,\vso}\colon \vec E(T)\to \vU$ be defined by setting
 $$\alpha'(\ve) := \alpha(\ve)\lor\vso\quad\text{if } \vex\le\ve$$%
   \COMMENT{}
   and letting $\alpha'(\ev):= \alpha'(\ve)^*$ for these~$\ve$. For example, we have $\alpha'(\vex) = \vso$ and $\alpha'(\evx) = \svo$.%
   \COMMENT{}

   \begin{figure}[htpb]
\centering
   	  \includegraphics{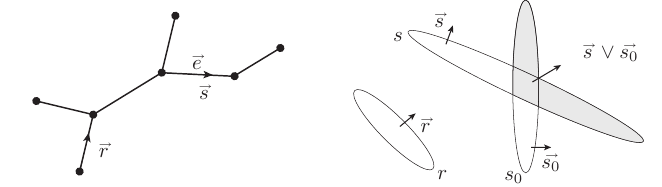}
   	  \caption{Shifting $\alpha(\ve) = \vs$ to $\alpha'(\ve) = \vs\vee\vso$}
   \label{fig:shiftsep}
   \end{figure}

Figure~\ref{fig:shiftsep} illustrates this for the case that $S$ consists of separations of a graph~$G$. If $\vso = (A,B)$, say, we would like to turn the separation $s$ of~$G$ into a separation `essentially' of~$G[B]$, the side of~$G$ to the right of~$s_0$ in the picture, that is nested with~$s_0$. There are two candidates for such a shift of~$s$: the separation $\vs\lor\vso$, which is indeed chosen, but potentially also the separation $\sv\lor\vso$. As long as $s$ is considered without a default orientation, it will not be clear which of these should be its shift. But, fortunately, $s$~has a default orientation in our scenario: we need shifts only of separations $s\in S$ that label an edge of our $S$-tree rooted at~$x$; and every {\em such} $s\in S$ can be oriented by default, as $\vs=\alpha(\ve)$ with $\ve\in T$ pointing away from~$x$, say. This is the orientation $\vs$ of~$s$ whose supremum with~$\vso$  we use in our definition of~$\alpha'$, a~choice which determines the shift not only of~$\vs$ but also of~$\sv$.

Informally, we think of a separation $\alpha'(\ve)$ as what becomes of $\alpha(\ve)$ when we `shift it across'~$\vso$. If $\alpha(\vex)$ is nontrivial and $(T,\alpha)$ is irredundant,%
   \COMMENT{}
   which will usually be the case, then these shifts $\alpha(\ve)\mapsto\alpha'(\ve)$ can be expressed by a map purely between separations in~$\vS$, without any reference to~$(T,\alpha)$. Our next aim is to define such a map.

Consider, instead of~$\vex\in\vec E(T)$ or $\alpha(\vex)$, an arbitrary nondegenerate {\em and nontrivial\/} separation~$\vr\in\vS$, and pick $\vso\in\vS$ with $\vr\le\vso$ as before. As $\vr$ is nontrivial%
   \COMMENT{}
   and nondegenerate, so is~$\vso$. Let $S_{\ge\vr}$ be the set of all separations $s\in S$ that have an orientation ${\vs\ge\vr}$. (Note that if $\vr=\alpha(\vex)$, with $(T,\alpha)$ an irredundant $S$-tree over stars rooted at~$x$, then $\vS_{\ge\vr}$ includes~$\alpha(\vec E(T))$ by Lemma~\ref{preservele}.) Since $\vr$ is nontrivial, only one of the two orientations $\vs$ of every $s\in S_{\ge\vr}\sm\{r\}$%
   \COMMENT{}
   satisfies $\vs\ge\vr$. Letting
 $$f\shift(\!\vr)(\vso) (\vs) := \vs\vee\vso\quad\text{and}\quad f\shift(\!\vr)(\vso) (\sv) := (\vs\vee\vso)^*$$
 for all $\vs \ge \vr$ in $\vSr\sm\{\rv\}$%
   \COMMENT{}
   thus defines a map $\vSr\to\vU$, the \emph{shifting map}~$f\shift(\!\vr)(\vso)$ (Fig.~\ref{fig:shiftsep}, right).%
   \COMMENT{}
   Note again that%
   \COMMENT{}
   $f\shift(\!\vr)(\vso)(\vr) = \vso$, since $\vr\le\vso$, and hence $f\shift(\!\vr)(\vso)(\rv) = \svo$.%
   \COMMENT{}
  In the case of $\vr = \vex$, with $(T,\alpha)$ an irredundant $S$-tree over stars rooted at~$x$, we then have 
 \begin{equation}\alpha' = f\shift(\!\vr)(\vso) \circ\ \alpha,\label{circ}
  \end{equation}
since for every edge $\ve\in\vec E(T)$ oriented away from~$x$ we have $\alpha(\ve)\in \vSr\sm\{\rv\}$, by Lemmas \ref{preservele} and~\ref{NewLemma},%
   \COMMENT{}
   so $f\shift(\!\vr)(\vso) $ maps $\alpha(\ve)$ to $\alpha(\ve)\lor\vso = \alpha'(\ve)$ and $\alpha(\ev) = \alpha(\ve)^*$ to its inverse $\alpha'(\ve)^* = \alpha'(\ev)$.

Our aim will be to show that $(T,\alpha')$ is another $S$-tree, and order-respecting if $(T,\alpha)$ is.%
   \COMMENT{}
   But we will need some assumptions to ensure this.

\begin{LEM}\label{lem:shifting}
Let $(T,\alpha)$ be an order-respecting $S$-tree, rooted at some leaf~$x$, and let $\vso\in\vS$ be nondegenerate and such that $\alpha(\vex)\le\vso$. Assume that $\alpha' = \alpha_{x,\vso}$ maps $\vec E(T)$ to~$\vS$. If $\vso$ is nontrivial in~$\vS$, or if the supremum in~$\vU$ of two separations that are trivial in~$\vS$ is never degenerate, then $(T,\alpha')$ is an order-respecting $S$-tree.
\end{LEM}

\begin{proof}
By definition, $\alpha'$~commutes with the involutions on $\vec E(T)$ and~$\vS$. Let us show that if $(T,\alpha)$ is order-respecting then~$\alpha'$, too, respects the ordering of~$\vec E(T)$. Consider edges $\ve < \vf$ in~$\vec E(T)$.%
   \COMMENT{}
   Suppose first that $\vex\le\ve$. Then $\alpha'(\vf) = \alpha(\vf)\lor\vso\ge\alpha(\ve)\lor\vso = \alpha'(\ve)$,%
   \COMMENT{}
   as desired. We now assume that $\vex\not\le\ve$.

If $\vex\le\fv$ then $\vex\le\fv < \ev$, which reduces to the case above on renaming $\fv$ as~$\ve$ and $\ev$ as~$\vf$. We may thus assume that $\vex\not\le\fv$, so that $\vex\le\vf$.

Now $\alpha'(\ve) = (\alpha(\ev)\lor\vso)^* = \alpha(\ve)\land\svo\le\alpha(\ve)\le\alpha(\vf)\le \alpha(\vf)\lor\vso = \alpha'(\vf)$ by~\eqref{deMorgan}.%
   \COMMENT{}

In order for $(T,\alpha')$ to be an $S$-tree it remains to show that $\alpha'(\ve)$ is nondegenerate for every $\ve\in\vec E(T)$. But if $\alpha'(\ve) = \alpha(\ve)\lor\vso$ is degenerate, it is distinct from both $\alpha(\ve)$ and~$\vso$ (which are nondegenerate by assumption)%
   \COMMENT{}
   and hence strictly greater than these. But then $\alpha(\ve)$ and $\vso$ are both trivial in~$\vS$,%
   \COMMENT{}
   so their supremum $\alpha'(\ve)$ is nondegenerate by assumption.
\end{proof}

In the premise of Lemma~\ref{lem:shifting} we assumed that $\alpha'$ maps $\vec E(T)$ to~$\vS$. Let us now define some conditions that ensure this. Let us say that $\vso\in\vS$ \emph{emulates $\vr\in\vU$%
   \COMMENT{}
   in~$\vS$} if $\vso\ge\vr$ and every $\vs\in\vS\sm\{\rv\}$%
   \COMMENT{}
   with $\vs\ge\vr$ satisfies $\vs\vee\vso\in\vS$.%
   \COMMENT{}
Applied with $\vr = \alpha(\vex)$ and $\vs = \alpha(\ve)$ for $\vex\le\ve\in\vec E(T)$, these conditions will help%
   \COMMENT{}
   ensure that $\alpha'$ takes $\vec E(T)$ to~$\vS$; see Lemma~\ref{lem:shifttree} below.

\begingroup\lineskiplimit=-2pt
Finally, we need a condition on~$\F$ to ensure that the shifts of multistars of separations associated with nodes of~$T$ are not only again multistars but also induce stars in~$\F$. Given any set $\F\sub 2^\vU\!$ of stars, let us say that a separation $\vso\in\vS$ \emph{emulates} $\vr\in\vU$ in~$\vS$ {\em for~$\F$\/} if $\vso$ emulates $\vr$ in~$\vS$%
   \COMMENT{}
    and for any star $\sigma\sub \vSr\sm\{\rv\}$ in~$\F$ that has an element $\vs\ge\vr$%
   \COMMENT{}
   we also have $f\shift(\!\vr)(\vso) (\sigma)\in\F$.%
   \COMMENT{}%
   \footnote{\label{FNweakerFlinked}In fact, we could make do with less: that $f\shift(\!\vr)(\vso)$ is defined (with image in~$\vS$, as now) only on some symmetric subset of $\vSr$ that contains~$\bigcup\F$, and that for some fixed $S'\sub S$ and every $\sigma$ as above we have $f\shift(\!\vr)(\vso)(\sigma)\cap\vSdash\in\F$ if $f\shift(\!\vr)(\vso)(\vs)\in\vSdash$ for the unique $\vs\ge\vr$ in~$\sigma$.%
   \COMMENT{}
   Informally, we need not insist that our $S$-trees over~$\F$ shift to $S$-trees over~$\F$ in their entirety, as long as those shifts of their separations $\alpha(\ve)$ that lie in some specified set $\vSdash\sub\vS$ contain an $S$-tree over~$\F$ between them. It will be easy to adapt the proof of Theorem~\ref{thm:strong} should this ever be necessary.%
   \COMMENT{}
   We shall get back to this right at the end, in Section~\ref{sec:essence}.}

\endgroup

We now have all the ingredients needed to shift a suitably prepared $S$-tree:%
   \COMMENT{}

\begin{LEM}\label{lem:shifttree}
  Let $\F\sub 2^\vU\!$ be a set of stars.%
   \COMMENT{}
   Let $(T,\alpha)$ be a tight%
   \COMMENT{}
   and irredundant\penalty-200\ ${S}$-tree over~$\F$ with at least one edge, rooted at a leaf~$x$.%
   \COMMENT{}
   Assume that $\vr :=\alpha(\vex)$ is nontrivial and nondegenerate, let $\vso\in\vS$ emulate $\vr$ in~$\vS$ for~$\F$,
    and consider $\alpha' := \alpha_{x,\vso}$. Then $(T,\alpha')$ is an order-respecting $S$-tree over~$\F\cup\{\{\svo\}\}$, in which $\{\svo\}$ is a star associated with~$x$ but with no other leaf of~$T$.
\end{LEM}

\begin{proof}
  By Lemma~\ref{preservele}, $(T,\alpha)$ is order-respecting. For every $\ve\in\vec E(T)$ with $\vex\le\ve$ we therefore have $\alpha(\ve)\ge\alpha(\vex)=\vr$, as well as $\alpha(\ve)\in\vS\sm\{\rv\}$ by Lemma~\ref{NewLemma}. As $\vso$ emulates~$\vr$ in~$\vS$, these two facts imply that $\alpha'$ takes $\vec E(T)$ to~$\vS$. As $\vr = \alpha(\vex)$ is nontrivial and nondegenerate, $\vso\ge\vr$ is nondegenerate and nontrivial. By Lemma~\ref{lem:shifting}, therefore, $(T,\alpha')$~is an order-respecting $S$-tree.

By definition of~$\alpha'$, we have $\alpha'(\vex) = \vr\lor\vso = \vso$. Hence, $x$~is associated in $(T,\alpha')$ with~$\{\svo\}$, which is a star since $s_0$ is nondegenerate. The other nodes $t$ of~$T$ are associated in~$(T,\alpha')$ with stars in~$\F$ because $\vso$ emulates~$\vr$ for~$\F$: note that $\alpha$ maps the stars $\vec F_t$ to $\vSr\sm\{\rv\}$ by Lemmas \ref{preservele}, \ref{NewLemma}%
   \COMMENT{}
   and~\ref{onlyedge},%
   \COMMENT{}
   that $\alpha(\vec F_t)$ contains a separation $\vs\ge\vr$ since $\vec F_t$ contains an edge~$\ve\ge\vex$, and that $\alpha' = f\shift(\!\vr)(\vso) \circ\ \alpha$ by~\eqref{circ}.%
   \COMMENT{}

Suppose finally that $\{\svo\}$ is associated in~$(T,\alpha')$ also with another leaf $y\ne x$ of~$T$, with incident edge $(t,y)$~say.%
   \COMMENT{}
   Let $\{\rvdash\}$ be associated with $y$ in~$(T,\alpha)$. Then $\vrdash \le \rv$ since $(T,\alpha)$ is order-respecting. If $\vr = \vso$ then $\alpha'=\alpha$, so $\{\svo\} = \{\rv\}$ is associated with $x$ and $y$ also in $(T,\alpha)$, contradicting Lemma~\ref{NewLemma}. Hence
 $$\vr < \vso = \alpha'(y,t) = (\alpha(t,y)\lor\vso)^* = (\rvdash\lor\vso)^* \specrel={\eqref{deMorgan}} \vrdash\land\svo \le \svo,$$
   where the first `=' holds by definition of~$y$, and the third by the definition of~$\vrdash$ and of $\alpha'$ based on the fact that $\vex\le (t,y)$.%
   \COMMENT{}
   Thus, $s_0$ witnesses that $\vr$ is trivial in~$\vS$, contrary to assumption.%
   \COMMENT{}
   \end{proof} 

Let us say that $\F$ {\it forces\/} the separations $\vs\in\vS$ for which $\{\sv\}\in\F$.%
   \COMMENT{}
   We say that $\vS$ is \emph{$\F$-separable}%
   \COMMENT{}
   if for every two nontrivial and nondegenerate%
   \COMMENT{}
   $\vr,\rvdash\in\vS$ that are not forced by~$\F$%
   \COMMENT{}
   and satisfy $\vr\le\vrdash$ there exists an $s_0\in S$ with an orientation~$\vso$ that emulates $\vr$ in~$\vS$ for~$\F$ and such that $\svo$ emulates $\rvdash$ in~$\vS$ for~$\F$. As earlier, any such~$s_0$ also be nondegenerate and have no trivial orientation.%
   \COMMENT{}%
   \footnote{However it can happen that $\vr < \rv\le\vso\ (\le\vrdash)$. Then $\rvdash\le\svo\le\vr$ as well as, by assumption, $\rvdash\le\rv$, so the nontriviality of~$\rvdash$ implies that $r=r'$. Then $\vr < \rv\le\vso\le\vrdash$ with equality in both cases, giving $\vr=\rvdash$.}%
   \COMMENT{}

We can now strengthen our weak duality theorem so as to yield consistent orientations, provided that $\vS$ is $\F$-separable. Recall that for a separation system $(\vS,\le\,,\!{}^*)$ and a set~$\F$,%
   \COMMENT{}
   an orientation $O$ of~$S$ is called an \emph{$\F$-tangle} if it is consistent and avoids~$\F$, that is, if $2^O\cap\F = \emptyset$. Let us call $\F$ \emph{standard} for~$\vS$ if it forces all $\vr\in\vS$ that are trivial in~$\vS$, i.e., contains the singleton stars~$\{\rv\}$ of their inverses.%
   \COMMENT{}

\begin{THM}[Strong Duality Theorem]\label{thm:strong} 
  Let $(\vU,\le\,,\!{}^*,\vee,\wedge)$ be a universe of  separations containing a finite separation system $(\vS,\le\,,\!{}^*)$. Let $\F\sub 2^\vU\!\!$ be a set of stars, standard for~$\vS$. If $\vS$ is $\F$-separable,%
   \COMMENT{}
   exactly one of the following assertions holds:\looseness=-1
\begin{enumerate}[\rm(i)]\itemsep0pt
  \item There exists an $\F$-tangle of~$S$.
  \item There exists an ${S}$-tree over $\F$.
  \end{enumerate}
\end{THM}

\noindent
   We remark that, by Lemma~\ref{prune}, the $S$-tree in~(ii) can be chosen irredundant, in which case it will be order-respecting (Lemma~\ref{preservele}).

\begin{proof}
  Since replacing $\F$ with $\F\cap 2^\vS$ leaves the validity of both (i) and (ii) unchanged we may, and shall, assume that $\F\sub 2^\vS$.%
   \COMMENT{}
   By Lemma \ref{lem:notboth}, (i) and (ii) cannot both hold; we show that at least one of them holds.

Since $\F$ is standard, the set
 $$O^- := \{\,\vs\mid \{\sv\}\in\F\,\}$$
 of separations that $\F$ forces contains all the trivial separations in~$\vS$. But it contains no degenerate ones, because the $\{\sv\}\in\F$ are stars. If $O^-\supe \{\vs,\sv\}$ for some~$s\in S$, then $(T,\alpha)$ with $T=K_2$ and ${\rm im}\,\alpha=\{\vs,\sv\}$ is an $S$-tree over~$\F$. We may thus assume that $O^-$ is antisymmetric: a partial orientation of $S\sm D$, where $D$ is the set of degenerate elements of~$S$. 

Let us show that $O^-$ is consistent. If not, then $O^-$ contains some $\rv$ and~$\vrdash$ such that $\vr < \vrdash$. As $O^-$ is antisymmetric, it then does not contain their inverses $\vr$ and~$\rvdash$.%
   \COMMENT{}
  So $\F$ does not force these; in particular, they are nontrivial.%
   \COMMENT{}
   Since $\vS$ is $\F$-separable, there exists an $s_0\in S$ with orientations~$\vso,\svo$ such that $\vso$ emulates $\vr$ in~$\vS$ for~$\F$ and $\svo$ emulates $\rvdash$ in~$\vS$ for~$\F$. Since $r$ is not degenerate%
   \COMMENT{}
   and $\vso$ emulates $\vr$ for~$\F$, the singleton star $\{\vr\}\in\F$%
   \COMMENT{}
   shifts to $\{f\shift(\!\vr)(\vso) (\vr)\} = \{\vso\}\in\F$, so $\svo\in{O^-}$. Likewise, since $r'$ is not degenerate and $\svo$ emulates~$\rvdash$ we have $\vso\in{O^-}$. This contradicts our assumption that ${O^-}$ is antisymmetric.

Let us show that $O^-\cup \vec D$ is still consistent. Suppose $r\ne s$ are such that $\rv,\vs\in O^-\cup\vec D$ and $\vr < \vs$. Then $r$ and $s$ are not both in~$D$, since that would imply $\vr < \vs = \sv < \rv = \vr$. Since $O^-$ is consistent, we may thus assume that $\rv\in O^-$ and $\vs\in\vec D$ (or vice versa, which is equivalent by~\eqref{invcomp}).%
   \COMMENT{}
   Then $\vr$ is trivial, as $\vr < \vs = \sv$. Hence $\vr\in O^-$ as well as, by assumption, $\rv\in O^-$.%
   \COMMENT{}
   This contradicts our assumption that $O^-$ is antisymmetric.

\medbreak

Let $R$ be the set of separations in $S\sm D$ of which neither orientation lies in~$O^-$. We shall apply induction on $|R|$%
   \COMMENT{}
   to show that (i) or~(ii) holds whenever $\F$ is such that $O^-$ is antisymmetric.%
   \COMMENT{}
   If $|R|=0$, then $O^-\cup\vec D$ is either itself an $\F$-tangle of~$S$ or contains a star $\sigma\in\F$. Then $\sigma\sub O^-$, since stars have no degenerate elements. By definition of~$O^-$, and since $O^-$ is antisymmetric, $\sigma$~is not a singleton subset of~$O^-$%
   \COMMENT{}
   (though it may be empty). Let $T$ be a star of $|\sigma|$ edges with centre~$t$, say, and let $\alpha$ map its oriented edges $(x,t)$ bijectively to~$\sigma$. Then $(T,\alpha)$ satisfies~(ii).%
   \COMMENT{}

For the induction step,%
   \COMMENT{}
   pick $\vro\in\vec R$.%
   \COMMENT{}
   Then neither $\vro$ nor $\rvo$ lies in~$O^-\cup\vec D$; let $\vrone\le\vro$ and $\rvtwo\le\rvo$ be minimal in~$\vS\sm (O^-\cup\vec D)$.%
   \COMMENT{}
   Then $\vrone\le\vro\le\vrtwo$. As $\F$ forces neither $\vrone$ nor~$\rvtwo$%
   \COMMENT{}
     and $\vS$ is $\F$-separable, there exists an $s_0\in S\sm D$%
   \COMMENT{}
   with nontrivial%
   \COMMENT{}
   orientations~$\vso,\svo$ such that $\vso$ emulates $\vrone$ in~$\vS$ for~$\F$ and $\svo$ emulates $\rvtwo$ in~$\vS$ for~$\F$.%
   \COMMENT{}%
   \COMMENT{}%
   \COMMENT{}

Since $O^-$ is antisymmetric, it does not contain both $\vso$ and~$\svo$. Let us assume that $\svo\notin O^-$, i.e.\ that $\{\vso\}\notin\F$. Then $\{\vrone\}\notin\F$, because $\vso$ emulates $\vrone$ for~$\F$ and $f\shift(\vrone)(\vso)$ maps the star $\{\vrone\}\sub \vSrone\sm\{\rvone\}$ to~$\{\vso\}$. Thus, $\vrone$ and $\rvone$ both lie outside $O^-\cup\vec D$, so $r_1\in R$.

We can now hope to apply the induction hypothesis to~$\F_1 := \F\cup\{\{\rvone\}\}$, because
 $$O^-_1 := O^-\cup \{\vrone\} =  \{\,\vs\mid \{\sv\}\in\F_1\,\}$$
 is again antisymmetric,%
   \COMMENT{}
   and the set $R_1$ of separations in $S\sm D$ with neither orientation in $O^-_1$ is smaller than~$R$. Also, $\F_1$~is a standard set of stars, because $\F$~is and $r_1\notin D$.%
   \COMMENT{}
   But we still have to check that $\vS$ is $\F_1$-separable. 

   \begingroup\lineskiplimit=-3pt
To do so, consider (nontrivial and) nondegenerate separations $\vr,\rvdash\in\vS$ not forced by~$\F_1$ such that $\vr\le\vrdash$. We have to find an $s_1\in S$ with an orientation~$\vsone$ that emulates $\vr$ in~$\vS$ for~$\F_1$ and such that $\svone$ emulates $\rvdash$ in~$\vS$ for~$\F_1$. By assumption, there is such an $s_1\in S$ for $\vr$ and $\vrdash$ with respect to~$\F$; let us take this~$s_1$, with orientations $\vsone,\svone$ such that $\vsone$ emulates $\vr$ for~$\F$ and $\svone$ emulates $\rvdash$ for~$\F$. We have to show that this emulation extends to~$\F_1$, i.e., that for the unique star $\{\rvone\}$ in $\F_1\sm\F$ we have $\{f\shift(\!\vr)(\vsone) (\rvone)\}\in\F_1$ if $\vr\le\rvone\ne\rv$%
   \COMMENT{}
   (so that $\{\rvone\}\sub \vSr\sm\{\rv\}$),%
   \COMMENT{}
   and $\{f\shift(\rvdash)(\svone) (\rvone)\}\in\F_1$ if $\rvdash\le\rvone\ne\vrdash$. In either case, the image $\sv$ of~$\rvone$ under the relevant map is either equal to~$\rvone$ (in which case we are done)%
   \COMMENT{}
   or greater,%
   \COMMENT{}
   by definition of the shift operator$\,\downarrow$. If $\rvone < \sv$, then $s\notin D$, since otherwise $\rvone < \vs=\sv$ would be trivial%
    \COMMENT{}
   and hence in~$O^-$.%
   \COMMENT{}
   And $\vs < \vrone$ by~\eqref{invcomp}, so $\vs\in O^-\cup\vec D$ by the minimality of $\vrone$ in~$\vS\sm (O^-\cup\vec D)$. Thus $\vs\in O^-$, and hence ${\{\sv\}\in\F\sub\F_1}$, by the definition of~$O^-$. This completes our proof that $\vS$ is $\F_1$-separable.

   \endgroup
We can thus apply the induction hypothesis to~$\F_1$. If it returns an $\F_1$-tangle of~$S$, then this is our desired $\F$-tangle. So we may assume that it returns an $S$-tree $(T_1,\alpha)$ over~$\F_1$. If this $S$-tree is even over~$\F$, our proof is complete. We may thus assume that $T_1$ has a leaf $x_1$ associated with~$\{\rvone\}$. We now apply Lemma~\ref{onlyedge} to prune and contract $(T_1,\alpha)$ to a tight and irredundant $S$-tree over~$\F_1$ that still contains~$x_1$. In this $S$-tree, which for simplicity we continue to call $(T_1,\alpha)$, no leaf other than~$x_1$ is associated with~$\{\rvone\}$ (Lemma~\ref{NewLemma}). Let $\alpha'_1 := \alpha_{x_1,\vso}$. By Lemma~\ref{lem:shifttree}, $(T_1,\alpha'_1)$ is an $S$-tree over $\F\cup\{\{\svo\}\}$, in which the star $\{\svo\}$%
   \COMMENT{}
   is associated with~$x_1$ but with no other leaf of~$T_1$. All the other nodes of $T_1$ are therefore associated with stars in~$\F$.%
   \COMMENT{}

If $\{\svo\}\in\F$,%
   \COMMENT{}
   then $(T_1,\alpha'_1)$ is in fact an $S$-tree over~$\F$, completing our proof. We may thus assume that $\{\svo\}\notin\F$, or equivalently that $\vso\notin O^-$. We can now use the induction hypothesis exactly as above (where we assumed that $\svo\notin O^-$),%
   \COMMENT{}
   considering~$\rvtwo$ in the same way as we just treated~$\vrone$, to obtain an irredundant $S$-tree $(T_2,\alpha'_2)$ over $\F\cup\{\{\vso\}\}$ in which $\{\vso\}$ is associated with a unique leaf~$x_2$, and all the other nodes are associated with stars in~$\F$.%
   \COMMENT{}

These trees can now be combined to the desired $S$-tree $(T,\alpha')$ over~$\F$ as in the proof of Theorem~\ref{thm:weak}:%
   \COMMENT{}
   add to the disjoint union $(T_2 - x_2)\cup (T_1 - x_1)$ the edge $y_2y_1$ between the neighbour $y_2$ of $x_2$ in $T_2$ and the neighbour $y_1$ of $x_1$ in~$T_1$, put $\alpha'(y_2,y_1) := \vso$ and $\alpha'(y_1,y_2) := \svo$, and otherwise let $\alpha'$ extend $\alpha'_1$ and~$\alpha'_2$.
\end{proof}

\section{\boldmath Essential $S$-trees and $\F$-tangles: a refinement}\label{sec:essence}

Let us return to the question of how much of a restriction is our condition in the premise of the strong duality theorem that $\F$ must be standard for~$\vS$, i.e., contain all co-trivial singletons, the stars~$\{\sv\}$ for which $\vs$ is trivial in~$\vS$. As noted before, any consistent orientation of~$S$, and hence any $\F$-tangle, will contain all trivial separations and hence avoid all these singletons. So adding them to~$\F$ will not change the set of $\F$-tangles.

But neither would removing them. Which thus seems like a good idea, if only to avoid unnecessary clutter.

Removing the co-trivial singletons from~$\F$ would, however, have an impact on the set of $S$-trees over~$\F$. The leaves of an $S$-tree over a standard~$\F$ can be associated with any co-trivial singleton star, but if we remove these stars from~$\F$ then such an $S$-tree will no longer be over~$\F$.

We might try to repair this by removing those leaves from our $S$-tree $(T,\alpha)$. The edge which such a leaf sends to its neighbour~$t$, however, maps to a separation~$\vr$ that would then be missing from the star $\alpha(\vec F_t)\in\F$ associated with~$t$, perhaps knocking it out of~$\F$. But as $\vr$ is trivial, its membership in $\alpha(\vec F_t)$ will not be the reason why we put $\alpha(\vec F_t)$ in~$\F$ in the first place: if the role of an $S$-tree over~$\F$ is to witness the nonexistence of an $\F$-tangle, then only the nontrivial separations in its stars are essential for that role. So let's try to delete all trivial separations from stars in~$\F$, and see if we can retain an $S$-tree over the modified~$\F$.

Given a separation system $(\vS,\le\,,\!{}^*)$ and $\F\sub 2^\vS$,%
   \COMMENT{}
   define the {\it essential core\/} $\F'$ of~$\F$ as
 $$\F' := \{\,\sigma\sm \vS^-\mid \sigma\in\F\,\},$$
 where $\vS^-\sub\vS$ is the set of all separations that are trivial in~$\vS$.%
   \COMMENT{}
   Note that if $\F$ is standard then so is~$\F'$,%
   \COMMENT{}
   since inverses of trivial separations are never trivial. Let us call an $S$-tree $(T,\alpha)$ {\em essential\/} if it is irredundant, tight, and $\alpha(\vec E(T))$ contains no trivial separation.

\begin{THM}\label{thm:essential}{\rm \cite{TreeSets}}
Let $(\vS,\le\,,\!{}^*)$ be a separation system and $\F\sub 2^\vS\!$ a set of~stars.
   \begin{enumerate}[\rm(i)]\itemsep=0pt\vskip-\smallskipamount\vskip0pt
   \item The $\F'$-tangles of~$S$ are precisely its $\F$-tangles.
   \item If $(T,\alpha)$ is any $S$-tree over~$\F$, there is an essential $S$-tree $(T',\alpha')$ over~$\F'$ such that $T'$ is a minor of~$T$ and $\alpha' = \alpha\!\restriction\!\vec E(T')$. Conversely, from any essential $S$-tree over~$\F'$ we can obtain an $S$-tree over~$\F$ by adding leaves, if $\F$ is standard for~$\vS$.
   \end{enumerate}
\end{THM}

\proof
(i) is immediate from the fact that $\F$-tangles, being consistent, contain all trivial separations and hence also avoid~$\F'$.%
   \COMMENT{}

For the first statement in (ii), let us start by making the given $S$-tree $(T,\alpha)$ irredundant by pruning it, as in Lemma~\ref{prune}. We then contract edges violating tightness, as explained before Lemma~\ref{onlyedge}. We finally make the resulting tree essential by deleting all its edges that $\alpha$ maps to trivial separations. This can be done recursively by deleting leaves associated with a co-trivial singleton: since $\F$ consists of stars, Lemma~\ref{preservele} implies that any $S$-tree over~$\F$ with an edge mapping to a trivial separation will also have such an edge issuing from a leaf. (Recall that if $\vs$ is trivial then so is every $\vr\le\vs$.) Pruning off leaves recursively in this way%
   \COMMENT{}
   will leave a well-defined tree at the end,%
   \COMMENT{}
   which has the properties desired for~$(T',\alpha')$.

For the second statement in (ii), let $(T,\alpha)$ be an essential $S$-tree over~$\F'$, and consider a node~$t\in T$. As $\alpha(\vec F_t)\in\F'$, there exists $\sigma\in\F$ such that $\sigma\sm\alpha(\vec F_t)$ is a set of trivial separations~$\vs$. For each of these add a new leaf, joining it to $t$ by an edge $\ve$ with $\alpha(\ve) := \vs$.
\endproof

Theorem~\ref{thm:essential} allows us to strengthen%
   \COMMENT{}
   each of the two alternatives in the strong duality theorem from its current version with the given set $\F$ of stars to an `essential' version with~$\F'$ instead. So why didn't we prove this stronger version directly?

The answer is pragmatic: this would have been possible, and we shall indicate in a moment how to do it. But it would have made the proof notationally more technical. As the proof stands, we need to allow inessential $S$-trees, because they can arise in the induction step when we combine two shifted $S$-trees, even if these were essential before the shift.

Indeed, recall what happens to the leaves of an $S$-tree when we shift it, by~$f\shift(\vr)(\vso)$ say. A~leaf, associated with~$\{\rvdash\}$, say, where $r\ne r'\ne s_0$ for simplicity, will be associated in the shifted tree with the star~$\{\rvdash\lor\vso\}$, because $\vr\le\rvdash$. But if $\vrdash < \vso$, as will frequently happen, then this star is a co-trivial singleton, because $\vrdash\land\svo < \vrdash$ as well as $\vrdash\land\svo < \svo < \rvdash$.

The way to overcome this problem is indicated in Footnote~\ref{FNweakerFlinked}, with $S'$ the set of separations in~$S$ that have no trivial orientation in~$\vS$. When we shift a star $\sigma\in\F$, its shift may contain trivial separations, but we could simply delete these to make the shifted $S$-tree essential, as in the proof of Theorem~\ref{thm:essential}(ii). To ensure that it is again over~$\F$, we would need to replace the current requirement in the definition of $\F$-separable, that $f\shift(\vr)(\vso)$ should map to~$\F$ any $\sigma\in\F$ that contains a separation $\vs\ge\vr$, with the requirement that for any such $\sigma\in\F$ we have $f\shift(\vr)(\vso)(\sigma)\cap \vSdash\in\F$ if $f\shift(\!\vr)(\vso)(\vs)\in\vSdash$.%
   \COMMENT{}

\section*{Acknowledgement}
We enjoyed a visit by Fr\'ed\'eric Mazoit to the University of Hamburg in the summer of 2013, when he gave a series of lectures on the bramble-based duality theorems for width parameters developed in~\cite{MazoitPartition,MazoitPushing}. We then tried, unsuccessfully, to apply these to obtain a duality theorem for $k$-blocks. This triggered the development of a similar unified duality theory for tangles instead of brambles, first in graphs and later more generally. Our original goal, a duality theorem for $k$-blocks, was eventually achieved in~\cite{ProfileDuality}, as an application of Theorem~\ref{thm:strong}.

The second author was supported by the Basic Science Research
Program through the National Research Foundation of Korea (NRF)
funded by  the Ministry of Science, ICT \& Future Planning
(2011-0011653)
and also by the Institute for Basic Science (IBS-R029-C1).

\bibliographystyle{abbrv}
\bibliography{collective}

\end{document}